\documentclass[10pt]{amsart}
\usepackage{latexsym}
\usepackage{amscd, amsfonts, eucal, mathrsfs, amsmath, amssymb, amsthm}
\input xy
\xyoption{all}
\usepackage{fullpage}
\usepackage{hyperref}
\usepackage{baskervald}
  
\usepackage{comment}

\newcommand{\field}[1]{\mathbf #1}
\newcommand{\mf}[1]{\mathfrak #1}
\newcommand{\mc}[1]{\mathcal #1}
\newcommand{\ms}[1]{\mathscr #1}
\newcommand{\widebar}[1]{\overline{#1}}

\newcommand{\R}{\field R}

\newcommand{\Z}{\field Z}

\newcommand{\simto}{\stackrel{\sim}{\to}}

\newcommand{\shom}{\ms H\!om}
\DeclareMathOperator{\Sh}{\bf Sh}
\DeclareMathOperator{\mSh}{Sh}

\newcommand{\send}{\ms E\!nd}
\DeclareMathOperator{\End}{End}

\newcommand{\Spec}{\operatorname{Spec}}
\newcommand{\spec}{\operatorname{Spec}}

\renewcommand{\P}{\field P}

\DeclareMathOperator{\Pic}{Pic}

\DeclareMathOperator{\pr}{pr}

\newcommand{\m}{\boldsymbol{\mu}}

\newcommand{\G}{\field G} 

\renewcommand{\H}{\operatorname{H}}

\DeclareMathOperator*{\tensor}{\otimes}

\DeclareMathOperator*{\ltensor}{\stackrel{\field L}{\otimes}}

\newcommand{\surj}{\twoheadrightarrow}

\newcommand{\inj}{\hookrightarrow}

\DeclareMathOperator{\Aut}{\operatorname{Aut}}

\DeclareMathOperator{\Isom}{\operatorname{Isom}}

\DeclareMathOperator{\Br}{\operatorname{Br}}

\DeclareMathOperator{\B}{\operatorname{B\!}}

\newcommand{\eps}{\varepsilon}

\DeclareMathOperator{\naive}{\bf NMO}

\DeclareMathOperator{\azumaya}{\bf BLT}

\newtheorem{lem}{Lemma}
\numberwithin{lem}{subsection}

\newtheorem{prop}[lem]{Proposition}

\newtheorem{cor}[lem]{Corollary}

\theoremstyle{definition}
\newtheorem{defn}[lem]{Definition}

\newtheorem{assumption}[lem]{Assumption}

\theoremstyle{remark}
\newtheorem{remark}[lem]{Remark}

\author{Rajesh S.~ Kulkarni}
\author{Max Lieblich}

\title{Blt Azumaya algebras and moduli of maximal orders}
\begin{document}
\begin{abstract}
  {We study moduli spaces of maximal orders in a ramified division algebra over the function field of a smooth projective surface.  As in the case of moduli of stable commutative surfaces, we show that there is a Koll\'ar-type condition giving a better moduli problem with the same geometric points: the stack of blt Azumaya algebras.  One virtue of this refined moduli problem is that it admits a compactification with a virtual fundamental class.}
\end{abstract}

\maketitle
\tableofcontents
\section{Introduction}
\label{sec:intro}

Much recent progress has been made on the structure theory of maximal orders
over algebraic surfaces.  Several authors have produced a satisfying minimal
model program for such orders (a sampling of which is represented by
\cite{MR2472162,MR2180454,thong} and their references).  Others have studied
the moduli of Azumaya orders in a fixed unramified division algebra and related
moduli problems (e.g.\ \cite{MR2309155,Lieblich20114145,MR2060023, MR2579390}).  

In this paper we extend the moduli theory to orders in a ramified Brauer class.
In so doing we encounter a phenomenon similar to that which occurs in the
moduli theory of stable projective surfaces, arising from an analogue of
Koll\'ar's condition on the compatibility of the reflexive powers of the
dualizing sheaf with base change.  Because the global dimension of our orders
is $2$, things are technically rather simpler than in Koll\'ar's theory, and we
arrive at a satisfying moduli space with a natural compactification carrying a
virtual fundamental class.  

As in the commutative theory, the na\"\i ve moduli problem (given by fixing the
properties of the fibers of a family) contains a refined version as a bijective
closed substack.  This refined moduli problem can be described as a moduli
problem of Azumaya algebras on stacks rather than orders on varieties.  (One
can also interpret this refined problem as a moduli theory of parabolic Azumaya
algebras.)  These Azumaya algebras have a precise interaction with the
ramification divisor arising from the structure of hereditary orders in matrix
algebras over discrete valuation rings, first described by Brumer, giving them
a structure we call \emph{Brumer log terminal\/}, or {\it blt}.

We begin in Section \ref{sec:normal-ord} by studying the local problem,
relating hereditary algebras over complete dvrs to Azumaya algebras over root
construction stacks.  This is globalized in Section \ref{sec:globalization}.  A
simple approach to families of maximal orders is described in Section
\ref{sec:naive-relat-maxim}.  The two resulting moduli problems are described
in Section \ref{sec:moduli-1} and compared in Section
\ref{sec:comparison-result} (with a proof that they can differ included in
Section \ref{sec:moduli-problems-are}).  The comparison relies crucially on
ideas similar to those introduced by Koll\'ar in his theory of hulls and husks
and a local analysis of reflexive Azumaya algebras on families of rational
double points.  Finally, in Section \ref{sec:gener-azum-algebr} we describe how
to compactify the Azumaya problem using algebra-objects of the derived category
of a stack (that one might think of as ``parabolic generalized Azumaya
algebras'') along lines familiar from \cite{MR2579390}, yielding a virtual
fundamental class.  We na\"\i vely hope that perhaps these classes will be
useful for defining new numerical invariants of terminal orders.

\section*{Acknowledgments}

The authors had helpful conversations with Daniel Chan, Paul Hacking, and Colin
Ingalls while working on this paper.  During the course of this work, the first
author was partially supported by NSF grants DMS-0603684, DMS-1004306 and DMS-1305377.  The
second author was partially supported by an NSF Postdoctoral Fellowship, NSF
grant DMS-0758391, NSF CAREER grant DMS-1056129, a Sloan Research Fellowship, and a University of Washington
Faculty Fellowship.

\section{Normal orders and parabolic Azumaya algebras} \label{sec:normal-ord}

\subsection{Hereditary orders over dvrs} \label{sec:hered-orders-over} Fix a
discrete valuation ring $R$ with uniformizer $t$ and residue field $\kappa$.
Fix a separable closure $\kappa\subset\widebar\kappa$.  Fix a positive integer
$n$ invertible in $R$.  Given a positive integer $r$, let $\pi:\ms X_r\to\spec
R$ be the stack-theoretic quotient of the natural action of $\m_r$ on
$R[s]/(s^r-t)$.  The root construction provides an isomorphism
$\B\m_{r,\widebar\kappa}\simto(\ms X_r\tensor_R\widebar\kappa)_{\text{\rm
red}}$.  An Azumaya algebra $\ms A$ on $\ms X_r$ thus gives rise to an Azumaya
algebra on $\B\m_{r,\widebar\kappa}$ by restriction.  By \S 4.1 of
\cite{paiitbgoaas}, any such algebra is isomorphic to the sheaf of
endomorphisms of the vector bundle on $\B\m_{r,\widebar\kappa}$ associated to a
representation of $\m_r$.  Call this the \emph{representation associated to
$\ms A$\/}; this representation is defined up to tensoring with a character.

\begin{defn}\label{D:type} Say that a hereditary order $A$ over $R$ is \emph{of
  type $m$\/} if $A\tensor R^{hs}$ has exactly $m$ distinct indecomposable
  projective modules, where $R^{hs}$ is the strict Henselization of $R$.  Given
  a positive divisor $m$ of $n$, call an Azumaya algebra $\ms A$ over $\ms X$
  \emph{of type $m$\/} if representation associated to $\ms A$ is the
  restriction of scalars of the regular representation of $\m_m$ via the
  natural quotient map $\m_n\to\m_m$.  \end{defn}

\begin{defn}
  The \emph{hereditary site} $\mc F$ of $\Spec R$ is the site whose underlying category
  consists of faithfully flat quasi-finite \'etale $R$-schemes $U\to \Spec R$
  with $U$ of pure dimension $1$, with coverings given by collections of $R$-maps
  $U_i\to U$ that are jointly surjective. 
\end{defn}
Define two stacks on the hereditary site of $\spec R$ as follows.

\begin{defn} Given an object $U\to\spec R$, an Azumaya algebra $\ms A$ on $\ms
  X_U$ is \emph{$n$-typed\/} if for each closed point $u\in U$ the restriction
  of $\ms A$ to $\ms X\tensor_R\ms O_{U,u}$ has type $m$ for some positive
  integer $m$ dividing $n$.  A hereditary order $A$ on $U$ is
  \emph{$n$-typed\/} if for every closed point $u\in U$, the restriction of $A$
  to $\ms O_{U,u}$ has type $m$ for some positive integer $m$ dividing $n$.
\end{defn}

\begin{defn} Given an object $U\to\spec R$ of $\mc F$, the stack $\mc A_n$ has
  as objects over $U$ the groupoid of $n$-typed Azumaya algebras $\ms A$ of
  degree $n$ on $\ms X\times_{\spec R}U$.  The stack $\mc H_n$ has as objects
  the groupoid of $n$-typed hereditary orders on $U$.  \end{defn}

Since the $n$-typed Azumaya and hereditary properties are \'etale-local, it is
clear that both $\mc A_n$ and $\mc H_n$ are stacks.

\begin{prop}\label{P:main-hered-az} For any object $\ms A\in\mc A_n(U)$, the
  finite $\ms O_U$-algebra $\pi_\ast\ms A$ lies in $\mc H_n$.  The resulting
  map of stacks $\mc A_n\to\mc H_n$ is a $1$-isomorphism.  \end{prop}
\begin{proof} Since both stacks are limit-preserving and the statements are
  \'etale-local on $U$, it suffices to prove the following: if $R$ above is a
  strictly Henselian discrete valuation ring then for any locally free sheaf
  $\ms V$ of rank $n$ and type $m$ on $\ms X$, the $R$-algebra
  $\pi_\ast\send(\ms V)$ is hereditary of type $m$, and in fact this gives an
  equivalence of groupoids between Azumaya algebras of degree $n$ and type $m$
  on $\ms X$ and hereditary $R$-algebras of degree $n$ and type $m$.  Indeed,
  since $\Br(K(R))[n]=0$, the generic fiber of any hereditary $R$-order and the
  Brauer class of any Azumaya algebra of degree $n$ over $\ms X$ are trivial,
  which reduces us to the case of matrix algebras and orders therein.

  We recall Brumer's fundamental description of hereditary orders
  \cite{MR0152565, MR0156872} (combined with Artin-de Jong characterization of
  the number of indecomposable projectives = number of embeddings in maximal
  orders): given a $K$-vector space $V$ of dimension $n$, the hereditary orders
  in $\End(V)$ of type $m$ are equivalent to collections of $R$-submodules
  $\{M_i\subset V\}_{i\in\Z}$ such that for all $i$ we have $M_{i+1}\subsetneq
  M_i$ and $M_{i+m}=tM_i$, up to a shift of indices.  The equivalence is given
  by sending $\{M_i\}$ to the ring of endomorphisms $f$ of $V$ such that for
  all $i$ we have $f(M_i)\subset M_i$; this filtered endomorphism ring is then
  the hereditary order corresponding to the filtered module $\{M_i\}$.

  On the other hand, Azumaya algebras of type $m$ on $\ms X_n$ are the
  pullbacks of Azumaya algebras $\ms A'$ of type $m$ on $\ms X_m$, and any such
  algebra $\ms A'$ is isomorphic to the pushforward of its pullback to $\ms
  X_n$ via the natural map $\ms X_n\to\ms X_m$.

  Thus, it suffices to prove the proposition in case $n=m$.  The filtered
  module $\{M_i\}$ is precisely an object of the category of parabolic vector
  bundles with denominator $n$, called $\operatorname{Par}_{\frac{1}{n}}(\spec
  R,(t))$ in \cite{MR2353089},  and the corresponding order is nothing other
  than the endomorphisms of the parabolic sheaf.  Just as in \cite{MR2353089},
  we know that there is a locally free sheaf $\ms V$ on $\ms X_n$ giving rise
  to $\{M_i\}$ in such a way that $\End(\ms V)$ equals the endomorphisms of the
  parabolic sheaf.  But the $R$-module $\End(\ms V)$ is precisely
  $\pi_\ast\send(\ms V)$.  

  What is $\ms V$?  Since each inclusion $M_{i+1}\subset M_i$ is proper, the
  eigendecomposition of $\ms V$ must have $n$ distinct summands, which implies
  that the representation associated to $\ms V$ is the regular representation.

  What are the automorphisms of $A:=\pi_\ast\send(\ms V)$?  Any
  $R$-automorphism of $A$ localizes to a $K$-automorphism of $\End(V)$, which
  by the Skolem-Noether theorem is given by conjugation by an automorphism
  $\phi$ of $V$.  If this conjugation is to preserve the set of morphism
  stabilizing the filtered module $\{M_i\}$ then $\phi$ itself must preserve
  the filtration, which means precisely that $\phi$ is induced by an
  automorphism of the parabolic sheaf corresponding to $\{M_i\}$, which in turn
  is equivalent to $\phi$ being induced by an automorphism of $\ms V$.  Thus,
  the induced map $\Aut(\send(\ms V))\to\Aut(A)$ is a bijection, as desired.
\end{proof}

The reader wishing to avoid stacks can also interpret the equivalence purely in
terms of parabolic sheaves: the hereditary orders on $R$ are equivalent (as a
groupoid) to ``parabolic Azumaya algebras'': parabolic sheaves of algebras
locally isomorphic to the parabolic sheaf of endomorphisms of a parabolic
vector bundle with denominator equal to the type of the order.  This seems to
hold no advantage (when the type is bounded as it is) over the formulation in
terms of root stacks.

\subsection{Globalization for terminal orders} \label{sec:globalization}

Let $\alpha$ be a terminal Brauer class over the function field of a smooth
surface $X$ in the sense of \cite{MR2180454}.  The ramification data of
$\alpha$ yield a simple normal crossings divisor $D=D_1+\cdots+D_m\subset X$,
and for each component $D_i$ a ramification degree $e_i|n$.  Let $\pi:\ms X\to
X$ be the smooth stack that is given by the fiber product (with respect to $i$)
of the root construction of order $e_i$ along $D_i$.  Let $\eta_i$ be the
generic point of $D_i$.  As in Section \ref{sec:hered-orders-over}, an Azumaya
algebra over $\ms X$ has associated representations over each
$\B\m_{e_i,\kappa(D_i)}$; call the representation associated to $D_i$ the
\emph{$i$th representation associated to $\ms A$\/}.

\begin{defn} An Azumaya algebra $\ms A$ on $\ms X$ is \emph{Brumer log terminal
  (blt)\/} if for every $i$ the local Azumaya algebra $\ms A_{\eta_i}$ has type
  $e_i$.  \end{defn}

Recall that a normal order with center $X$ and Brauer class $\alpha$ is called
\emph{terminal\/} in the notation of \cite{MR2180454}.

\begin{prop}\label{P:push forward get terminal} Pushforward by $\pi$ defines an
  equivalence of groupoids between blt Azumaya algebras on $\ms X$ and terminal
  orders on $X$ with Brauer class $\alpha$.  \end{prop} \begin{proof} The proof
  is mainly a routine globalization of Proposition \ref{P:main-hered-az}.  

  First, we have that the pushforward of any such $\ms A$ is normal, as we can
  check this locally at any codimension $1$ point, where this is an immediate
  consequence of Proposition \ref{P:main-hered-az}.  Thus, the pushforward of a
  blt Azumaya algebra is normal, as desired.  To show that $\pi_\ast$ is
  essentially surjective, note that since any maximal order $A$ is reflexive we
  have that $$A=\bigcap_{x\in X^{(1)}}A_x,$$ and similarly for blt Azumaya
  algebras on $\ms X$, where $X^{(1)}$ is the set of codimension $1$ points of
  $X$ and $A_x:=A\tensor\ms O_{X,x}$ is the localization.  Moreover, $\pi_\ast$
  commutes with the formation of intersections.  It thus suffices to prove the
  analogous result for localizations at codimension $1$ points (keeping track
  of the embedding in the generic algebras), which is precisely Proposition
  \ref{P:main-hered-az}.

  To show that $\pi_\ast$ is fully faithful, it suffices to prove the analogous
  statement upon replacing $X$ by its localization at $\eta_i$.  Indeed, since
  the maximal orders $A$ and the Azumaya algebras $\ms A$ are reflexive, we
  have that for any blt Azumaya algebras $\ms A$ and $\ms B$ with pushforwards
  $A$ and $B$ the isomorphisms are given by $$\Isom(\ms A,\ms B)=\bigcap_{x\in
    X^{(1)}}\Isom(\ms A_x,\ms B_x)$$ and $$\Isom(A,B)=\bigcap_{x\in
      X^{(1)}}\Isom(A_x,B_x)$$ where $X^{(1)}$ is the set of codimension $1$
      points of $X$ and the intersection takes place inside the set $\Isom(\ms
      A_\eta,\ms B_\eta)$ of isomorphisms of the generic algebras.  Since
      Propostion \ref{P:main-hered-az} shows that $\Isom(\ms A_x,\ms
      B_x)=\Isom(A_x,B_x)$, the result follows.  \end{proof}

In more classical terms, terminal orders are parabolic Azumaya algebras with
parabolic structure along the ramification divisor.

\section{Na\"ive relative maximal orders} \label{sec:naive-relat-maxim}

\subsection{Definitions and basic geometric properties}
\label{sec:defin-basic-geom}

\begin{defn} Let $Z$ be an integral algebraic space.  A torsion free coherent
  sheaf $\ms A$ of $\ms O_Z$-algebras is a \emph{maximal order} if any
  injective morphism $\ms A\to\ms B$ of torsion free $\ms O_Z$-algebras that is
  an isomorphism over a dense open subspace $U\subset Z$ is an isomorphism.
\end{defn}

We will prove that maximality in a family is a fiberwise condition.

\begin{defn} Given a morphism $X\to S$ with locally Noetherian geometric
  fibers, an \emph{$S$-flat family of coherent sheaves\/} is an $S$-flat
  quasi-coherent $\ms O_X$-module $\ms F$ of finite presentation.  If $X$ has
  integral fibers, we will say that a possibly non-flat quasi-coherent $\ms
  O_X$-module of finite presentation $\ms G$ is \emph{torsion free\/} if its
  geometric fibers $\ms G_s$ are torsion free coherent $\ms O_{X_s}$-modules.
\end{defn}

\begin{defn} Given a flat morphism $X\to S$ with integral fibers, an $S$-flat
family of coherent $\ms O_X$-algebras $\ms A$ is \begin{enumerate} \item a
      \emph{relative maximal order\/} if for any $T\to S$ and any injective
      morphism $\ms A_T\to\ms B$ into a torsion free $\ms O_{X_T}$-algebra that
      is an isomorphism over a fiberwise dense open subspace $U\subset X_T$ is
    an isomorphism; \item a \emph{relative normal order\/} if the geometric
      fibers $\ms A$ are $R_1$ and $S_2$, in the sense of \cite{MR2180454}.
  \end{enumerate} \end{defn}

While relative normality is defined as a fiberwise condition, relative
maximality is not obviously so.  Let us prove this.

\begin{lem}\label{L:algclbase} Suppose $X$ is a proper integral algebraic space
  over an algebraically closed field $k$.  A coherent sheaf $\ms A$ of $\ms
  O_X$-algebras is a maximal order on $X$ if and only if it is a relative
  maximal order on $X/\spec k$.  In particular, for any field extension $K/k$
  we have that $\ms A\tensor K$ is a maximal order on $X\tensor K$.  \end{lem}
\begin{proof} Since any relatively maximal order is obviously maximal, it
  suffices to assume that $\ms A$ is maximal and prove that it is relatively
  maximal.  Suppose $\ms A_T\to\ms B$ is an injective map to a torsion free
  $\ms O_{X_T}$-algebra that is an isomorphism over the fiberwise dense open
  $U\subset X_T$.  For any geometric point $\spec K\to T$, the base change $\ms
  A_K\to\ms B_K$ is thus injective and an isomorphism over a dense open of the
  scheme $X_K$.  If we can show that this restricted map is always an
  isomorphism then the result is proven.  Thus, we are reduced to the case in
  which $T=\spec K$ with $K$ an algebraically closed extension field of $k$.

  Since $\ms B$ is of finite presentation, we may assume by a standard limit
  argument that there is a finite type integral $k$-scheme $T'\to\spec k$, a
  torsion free algebra $\ms B'$ over $T'$ with an injective map $\phi:\ms
  A_{T'}\to\ms B'$, and a dominant morphism $\spec K\to T'$ such that the base
  change of $\phi$ isomorphic to the given inclusion $\ms A_T\to\ms B$.  The
  locus over which $\phi$ is an isomorphism is an open subscheme $U'\subset
  X_{T'}$ whose restriction to the geometric generic fiber over $T'$ is
  non-empty.  By Chevalley's theorem the image of $U'$ in $T'$ is
  constructible, hence contains a dense open, whence shrinking $T'$ we may
  assume that $U'$ is dense in every fiber.  But now $T'$ has a dense set of
  $k$-points (as it is of finite type over an algebraically closed field), and
  we know by assumption that for any such point $t'\in T'$ the restriction $\ms
  A_{t'}\inj\ms B_{t'}$ is an isomorphism.  We conclude that $U'=T'$, which
  finishes the proof that $\ms A$ is a relative maximal order.  \end{proof}

\begin{remark}\label{R:base field funny} Note that if the base field $k$ is not
  assumed to be algebraically closed, the result of Lemma \ref{L:algclbase} is
  false.  Indeed, there are Brauer classes on varieties $X$ over a field $k$
  which are ramified but become unramified over the algebraic closure of $k$.
  Any maximal order over $k$ will be geometrically hereditary but non-maximal
  at the generic points of the preimage of the ramification divisor in
  $X\tensor\widebar k$.  A simple example is furnished by the quaternion
  algebra $(x,a)$ over $k(x,y)$, where $a$ is a non-square element of $a$.
  This gives a ramified algebra on $\P^2$ whose base change to $\widebar k$ is
  trivial, and it follows that no maximal order can be relatively maximal.
\end{remark}

\begin{prop}\label{P:rmaxope} Suppose $X\to S$ is a flat morphism of finite
  presentation between algebraic spaces whose geometric fibers are integral.
  An $S$-flat family of torsion free coherent $\ms O_X$-algebras $\ms A$ is a
  relative maximal order if and only if for every geometric point $s\to S$ the
  fiber $\ms A_s$ is a maximal order on the integral $\kappa(s)$-space $X_s$.
\end{prop} \begin{proof} It follows immediately from the definition that the
  geometric fibers of a relative maximal order are maximal.  To prove the other
  implication, by Lemma \ref{L:algclbase} it suffices to assume that the
  geometric fibers are maximal and show that $\ms A$ is maximal (i.e., we may
  assume that $T=S$; lifting geometric points to $T$ by taking field extensions
  does not disturb the hypotheses by Lemma \ref{L:algclbase}).

  Suppose $\iota:\ms A\to\ms B$ is an injection into a torsion free $\ms
  O_X$-algebra that is an isomorphism over a fiberwise dense open $U\subset X$.
  To prove that $\iota$ is an isomorphism it suffices to work locally on $S$,
  so we can assume that $S=\spec A$ for $A$ a local ring whose closed point $s$
  is the image of a geometric point over which $\ms A$ is maximal.  Since
  $\iota$ is an isomorphism over a fiberwise dense open and $\ms A$ and $\ms B$
  have torsion free fibers, the reduction $\iota_s:\ms A_s\to\ms B_s$ is
  injective and an isomorphism over a dense open.  Since $\ms A_s$ is maximal
  (as follows immediately from the same being true of its base change to
  $\widebar\kappa(s)$), we conclude that $\iota_s$ is an isomorphism.  By
  Nakayama's Lemma, we have that $\iota$ is surjective, whence it is an
  isomorphism, as desired.  \end{proof}

\begin{cor} Suppose $X$ is a smooth projective surface over a field $k$ and $D$
  is a central division algebra over its function field.  Let $k\to R\surj k$
  be an Artinian $k$-algebra with residue field $k$.  Given a maximal order
  $\ms A\subset D$, any infinitesimal deformation of $\ms A$ over $X\tensor_k
  R$ is a maximal order in the generic algebra $D\tensor_k R$.  \end{cor}
\begin{proof} There's only one geometric fiber!  \end{proof}

\begin{prop}\label{P:overlocring} Suppose $X\to S=\Spec A$ is an algebraic
  space of finite presentation with integral fibers over a local ring $A$ with
  residue field $\kappa$.  An $A$-flat family of torsion free $\ms
  O_X$-algebras $\ms A$ is a relative maximal order if and only if its
  geometric closed fiber is a maximal order on $X\tensor\widebar\kappa$.
\end{prop} \begin{proof} We may suppose that $A$ is Noetherian and reduced.  By
  Proposition \ref{P:rmaxope}, it suffices to prove that the geometric fibers
  are all maximal, which immediately reduces us by a pullback argument and
  Lemma \ref{L:algclbase} to showing that if $A$ is a discrete valuation ring
  with algebraically closed residue field then the generic fiber of $\ms A$ is
  a maximal order (in the absolute sense).  

  Let $\ms A_\eta\inj\ms B_\eta$ by an injection into a torsion free $\ms
  O_{X_\eta}$-algebra that is an isomorphism over the generic point of $\ms
  X_\eta$.  Let $\gamma\in X$ be the generic point of the closed fiber and let
  $\delta\in X$ be the generic point of $X$.  Considering localizations as
  quasi-coherent sheaves on $X$, we can focus on quasi-coherent sheaves of
  algebras containing $\ms A$ whose localizations at $\gamma$ are isomorphic to
  $\ms A_\gamma$ via the natural inclusion.  A standard argument shows that
  there is a coherent such algebra $\ms B$ extending $\ms B_\eta$; saturating
  if necessary, we may assume that $\ms B$ has torsion free fibers.  This
  produces a family $\ms A\inj\ms B$ over all of $X$ which is an isomorphism
  over a fiberwise dense open subscheme.  Reducing to $\kappa$ as in the proof
  of Proposition \ref{P:rmaxope}, we conclude that $\ms A\to\ms B$ is an
  isomorphism, whence the original map $\ms A_\eta\inj\ms B_\eta$ is an
  isomorphism, showing that $\ms A_\eta$ is maximal.  (Applying the same
  argument to a localization of the normalization in any extension of the
  fraction field of $A$ shows that the geometric generic fiber of $\ms A$ is
  maximal.) \end{proof}

Let $f:Z\to S$ be a flat morphism of finite presentation between algebraic
spaces with integral geometric fibers and $\ms A$ an $S$-flat torsion free $\ms
O_Z$-algebra of finite presentation.  Define a subfunctor
$\operatorname{Az}_{\ms A}\subset Z$ parametrizing morphisms $T\to Z$ such that
$\ms A_T$ is Azumaya.  \begin{lem} The map of functors $\operatorname{Az}_{\ms
  A}\inj Z$ is a quasi-compact open immersion.  \end{lem} \begin{proof} By
  absolute Noetherian approximation, there is an algebraic space $S_0$ of
  finite type over $\Z$, flat morphism $Z_0\to S_0$ of finite type with
  integral geometric fibers, and a morphism $S\to S_0$ such that the pullback
  of $Z_0$ to $S$ is isomorphic to $Z$.  Since $\ms A$ is of finite
  presentation, we can assume that $\ms A$ is defined on $Z_0$.  Now, since
  $Z_0$ is Noetherian any open subscheme is quasi-compact.  Thus, it suffices
  to prove that $\operatorname{Az}_{\ms A}\inj Z$ is open to conclude that it
  is quasi-compact.

  Since the locus over which $\ms A$ is locally free is open and contains
  $\operatorname{Az}_{\ms A}$, we may shrink $Z$ and assume that $\ms A$ is
  locally free.  Consider the morphism of locally free sheaves $\mu:\ms
  A\tensor\ms A^\circ\to\send(\ms A)$ given by left and right multiplication.
  We know that $\ms A_T$ is Azumaya if and only if $\mu_T$ is an isomorphism,
  identifying $\operatorname{Az}_{\ms A}$ with the functor of points on which
  $\mu$ is an isomorphism.  But this is equivalent to the cokernel of $\mu$
  vanishing, which is clearly an open condition.  \end{proof}

By Chevalley's theorem, the image of $\operatorname{Az}_{\ms A}$ in $S$ is a
constructible set $\operatorname{gAz}_{\ms A}\subset |S|$.

\begin{defn} The set $\operatorname{gAz}_{\ms A}$ will be called the
  \emph{central simple locus\/} of $\ms A$.  \end{defn}

The constructible central simple locus has two nice properties.  First, it is
open.

\begin{prop}\label{P:csl-open} Let $Z\to S$ be a proper morphism of finite
  presentation between algebraic spaces with integral geometric fibers.  Given
  a relative maximal order $\ms A$ on $Z$, the central simple locus of $\ms A$
  is open.  \end{prop} \begin{proof} Since the formation of
  $\operatorname{gAz}_{\ms A}$ is compatible with base change and $\ms A$ is of
  finite presentation, we immediately reduce to the case in which $S$ is
  Noetherian.  Now, since $\operatorname{gAz}_{\ms A}$ is constructible, to
  show that it is open it suffices to prove it under the additional assumption
  that $S=\spec R$ is the spectrum of a discrete valuation ring and that
  $\operatorname{gAz}_{\ms A}$ contains the closed point of $S$.  Let $\eta$ be
  the generic point of the closed fiber of $Z$ over $S$.  The localization $\ms
  A_\eta$ is a finite flat algebra over the discrete valuation ring $\ms
  O_{Z,\eta}$.  (The latter is a dvr because the fiber is integral, so the
  uniformizing parameter on $S$ is also a uniformizer in $\ms O_{Z,\eta}$.)
  Moreover, the reduction $\ms A\tensor\kappa(\eta)$ is a central simple
  algebra.  Thus, the closed fiber of the map $\ms A_\eta\tensor\ms
  A_\eta^\circ\to\send(\ms A_\eta)$ of free $\ms O_{Z,\eta}$-modules is an
  isomorphism.  By Nakayama's Lemma, the generic fiber is also an isomorphism,
  which shows that the generic stalk of $\ms A$ is a central simple algebra
  over the function field of $Z$, as desired.  \end{proof}

Second, fixing a Brauer class yields a closed central simple locus, in the
following sense.

\begin{prop} Suppose $X$ is a variety over a field $k$ and $S$ is a $k$-scheme.
  Let $\ms A$ be a relative maximal order on $X\times S$.  Suppose there exists
  a class $\alpha\in\Br(k(X))$ such that for every geometric point
  $s\in\operatorname{gAz}_{\ms A}$ the restriction of $\ms A_s$ to
  $\kappa(s)(X)$ has Brauer class $\alpha$.  Then the central simple locus
  $\operatorname{gAz}_{\ms A}$ is closed in $S$.  \end{prop} \begin{proof} We
  immediately reduce to the case in which $S$ is Noetherian.  Since
  $\operatorname{gAz}_{\ms A}$ is constructible and compatible with base change
  on $S$, and relative maximal orders are stable under base change, to show
  that $\operatorname{gAz}_{\ms A}$ is closed it suffices to prove it under the
  additional assumption that $S=\spec R$ is the spectrum of a dvr and
  $\operatorname{gAz}_{\ms A}$ contains the generic point.  Let $\eta$ be the
  generic point of the closed fiber of $X\times S$.  Given an inclusion of
  finite algebras $\iota:\ms A_\eta\inj B$, there is an $S$-flat coherent sheaf
  of $\ms O_{X\times S}$-algebras $\ms B$ with an inclusion $\ms A\inj\ms B$
  whose germ over $\eta$ is isomorphic to $\iota$.  Indeed, the subsheaf
  $B\subset\ms A_{K(X)}$ is a colimit of the finite algebras that contain $\ms
  A$, and some member of the directed system will have stalk $B$ at $\eta$.  

  It follows that $\ms A_\eta$ is a maximal order in its fraction ring $F:=\ms
  A_\eta\tensor K(X)$.  But we know that $F$ is a central simple algebra with
  Brauer class restricted from $\ms O_{X\times S,\eta}$, and therefore that any
  maximal order over $\ms O_{X\times S,\eta}$ in $F$ is Azumaya.  It follows
  that $\ms A_\eta$ is Azumaya, and therefore that $\operatorname{gAz}_{\ms A}$
  contains the closed point of $S$, as desired.  \end{proof}

Finally, let us define a relative terminal order of relative global dimension
$2$.  Suppose $S$ is an algebraic space and $Z\to S$ is a proper smooth
relative surface.  Suppose furthermore that $R=R_1+\cdots+R_m$ is a(n $S$-flat)
relative snc divisor on $Z$.

\begin{defn} A Brauer class $\alpha\in\Br(Z\setminus R)$ is \emph{terminal\/}
  if its restriction to every geometric fiber $Z_s$ is terminal in the sense of
  Definition 2.5 of \cite{MR2180454} and for each $i$ the ramification index
  $e_i(s)$ of $\alpha$ along $(R_i)_s$ is independent of $s$.

  A relative maximal order $\ms A$ on $Z$ with Brauer class $\alpha$ will be
  called a \emph{relative terminal order\/}.  \end{defn}

When working over a non-algebraically closed field, the pathology of Remark
\ref{R:base field funny} remains an issue: given a Brauer class
$\alpha\in\Br(k(X))$ that is ramified but such that its base change to
$\widebar k$ is unramified, no maximal order $\ms A$ with class $\alpha$ will
be relatively maximal over $k$ (because it is not geometrically maximal).  The
order $\ms A$ is still relatively normal, however.  Thus, if one endeavors to
study moduli spaces associated to Brauer classes such as $\alpha$, one should
allow certain normal orders.  Of course, one would not like to allow arbitrary
normal orders in a given division algebra, only those orders whose non-Azumaya
locus is related to the ramification locus of $\alpha$ over the base field.  

When the base field is algebraically closed this pathology does not happen, as
one cannot dissolve ramification with a base extension.  We will focus our
attention on this case in the present paper.

\section{Moduli} \label{sec:moduli-1}

\subsection{Notation and assumptions} In this section $X\to S$ will denote a
proper smooth relative surface of finite presentation and $D=D_1+\cdots+D_r$
will be a fixed relative snc divisor in $X$.  This means that each $D_i$ is a
proper smooth relative curve over $S$ and that for any pair $i\neq j$ the
intersection scheme $D_i\cap D_j$ is finite \'etale over $S$.  We also fix a
class $\alpha\in\Br(U)[n]$, where $U=X\setminus D$ and $n$ is invertible on
$S$.  In this section we will try to describe moduli of maximal orders with
Brauer class locally (on $S$) equal to $\alpha$.

\begin{assumption}\label{A:ass} There are integers $e_1,\ldots,e_r>1$ such that
  for each geometric point $s\to S$, the fiber $\alpha|_{U_s}$ is ramified to
  order $e_i$ on $D_i$, and this ramification configuration is terminal in the
  sense of Definition 2.5 of \cite{MR2180454}.  \end{assumption}

Note that the pair $(X,\Delta)$ with $\Delta:=\sum(1-\frac{1}{e_i})D_i$
associated to the ramification datum is Kawamata log terminal.  This appears to
be the genesis of this notation.

A simple example the reader should keep in mind is when $S$ is the spectrum of
an algebraically closed field and $\alpha$ is a Brauer class with snc
ramification divisor $D=D_1+\cdots+D_r$.  Our more general setup gives us the
ability to work with families of such Brauer classes, but a proper theory would
allow singular fibers of $X/S$.

There are two moduli problems that one can associate to the pair
$(X/S,\alpha)$.  



\subsection{Na\"ive families} \label{sec:naive-families} In this section we
write $\bf A$ for the stack of $S$-flat torsion free coherent algebras on $X$.
As described in \cite{MR2233719}, $\bf A$ is an Artin stack locally of finite
presentation over $S$.

\begin{defn}\label{D:naive} The stack of \emph{na\"ive maximal orders\/} is the
  stack $\naive_{X/S}^{\alpha}$ whose objects over an $S$-scheme $T$ are
  relative maximal orders $\ms A$ on $X\times_S T$ such that for every
  geometric point $t\to T$ the Brauer class of $\ms A|_{U\times_T t}$ equals
  $\alpha|_{U\times_S t}$.  \end{defn}

\begin{remark} One might think that in Definition \ref{D:naive} one should
  require that the Brauer class is $\alpha$ \'etale-locally on the base.  As we
  will see in Section \ref{sec:moduli-problems-are}, this does not materially
  improve the situation.  \end{remark}

\begin{lem}\label{L:pregrothex} Let $A$ be a local Noetherian ring over $S$, and let
  $\ms A$ be a flat family of coherent $\ms O_{X}$-algebras over $A$. If the
  closed fiber of $\ms A$ belongs to $\naive_{X/S}^\alpha$ then so does $\ms
  A$.  
\end{lem} 
\begin{proof} By Proposition \ref{P:overlocring} $\ms A$ is a
  relative maximal order, and the usual characterizations show that $\ms A$ is
  Azumaya over $U_A$.  It remains to show that for any geometric fiber of $X$
  over $A$ the Brauer class of that fiber of $\ms A$ is $\alpha$.  It suffices
  to prove this under the assumption that $A$ is a complete discrete valuation
  ring.  Thus, we may assume that $X_A$ is a regular scheme of dimension $3$
  and $\ms A$ is a maximal order which is Azumaya away from a snc divisor
  $D=D_1+\cdots+D_r$ and whose Brauer class has order invertible in $A$.  For
  sufficiently large and divisible $N$, the Brauer class of $\ms A_U$ extends
  to an element of $\beta$ in the Brauer group of the root construction
  $X\{D^{1/N}\}$ (in the notation of \cite{MR2753611}).  By the proper base
  change theorem for the morphism $X\{D^{1/N}\}\to\spec A$, the class $\beta$
  is determined by its closed fiber, so it must equal the pullback of $\alpha$,
  whence the geometric generic fiber of $\ms A_U$ has Brauer class $\alpha$, as
  desired.  \end{proof}

\begin{cor}\label{C:grothex} Let $A$ be a complete local ring with maximal
  ideal $\mf m$.  The functor
  $$\naive_{X/S}^\alpha(A)\to\lim_n\naive_{X/S}^\alpha(A/\mf m^{n+1})$$ is an
  equivalence of categories.  
  \end{cor} 
  \begin{proof} This is the classical
  Grothendieck existence theorem combined with Proposition \ref{P:overlocring}
  and Lemma \ref{L:pregrothex}, which says that the effectivization of any
  formal family lying in $\naive_{X/S}^\alpha$ also lies in
  $\alpha_{X/S}^\alpha$.  
  \end{proof}

\begin{lem}\label{L:hered-open}
Suppose $\ms A$ is a flat family of coherent $\ms O_X$-algebras over a Noetherian base scheme $T$ that is of finite type over an excellent Dedekind domain or a field. There is an open subscheme $U\subset T$ such that for any geometric point $t\to T$, the geometric fiber $\ms A_t$ is in $\naive_{X/S}^\alpha$ if and only if $t$ factors through $U$.
\end{lem}
\begin{proof}
By Theorem 0.5 of \cite{MR1728399}, it suffices to prove the result after replacing $T$ by a Dedekind scheme, and now we wish to show that the geometric generic fiber of $\ms A$ is in $\naive_{X/S}^\alpha$ if and only if all but finitely many geometric fibers lie in $\naive_{X/S}^\alpha$. By Lemma \ref{L:pregrothex}, if any closed geometric fiber is in $\naive_{X/S}^\alpha$ then the geometric generic fiber is in $\naive_{X/S}^\alpha$. It thus suffices to show that if the geometric generic fiber is in $\naive_{X/S}^\alpha$ then all but finitely many geometric closed fibers are in $\naive_{X/S}^\alpha$.

By Proposition \ref{P:push forward get terminal}, the generic fiber $\ms A_\eta$ is the pushforward of a blt Azumaya algebra $\mf A_\eta$ on $\ms X_\eta$ along the morphism $\pi_\eta:\ms X_\eta\to X_\eta$. By spreading out, we may assume after removing finitely many points from $T$ that $\mf A$ extends to all of $\ms X$. Moreover, the isomorphism $\ms A_\eta\simto\pi_\ast\mf A_\eta$ extends to an isomorphism over some dense open $U\subset X$ that contains the generic fiber. The complement of $U$ will have finite image in $T$, whereupon we have identifed the remaining fibers with pushfowards of blt Azumaya algebras with Brauer class $\alpha$, rendering them elements of $\naive_{X/S}^\alpha$, as desired.
\end{proof}

\begin{prop} The stack $\naive_{X/S}^\alpha$ is an Artin stack locally of
  finite presentation over $S$, and the morphism $\naive_{X/S}^\alpha\to\bf A$
  is an open immersion.  
\end{prop} 
\begin{proof} 
It suffices to prove the latter statement by \href{http://stacks.math.columbia.edu/tag/01TQ}{Tag 01TQ} of \cite{stacks-project}. Since belonging to $\naive_{X/S}^\alpha$ is a fiberwise statement, this follows immediately from Lemma \ref{L:hered-open}.
\end{proof}

We arrive at the somewhat surprising conclusion that maximal orders with Brauer
class $\alpha$ form an open substack of the stack of all coherent algebras.
However, the deformation theory is ``arbitrarily bad'' in the sense that it is
identical to the deformation theory of maximal orders.  We will describe a
refinement of the moduli problem with the same closed points but different
infinitesimal properties that has a natural compactification admitting a
virtual fundamental class.

\begin{remark} Without the Assumption \ref{A:ass}, the openness of the locus of
  na\"ive families is undoubtedly false.  \end{remark}

\subsection{Blt Azumaya families} \label{sec:azumaya-families}

Write $\pi:\widetilde X\to X$ for the stack $X\langle D^{1/n}\rangle$ in the
notation of Section 3.B of \cite{MR2753611}; the stack $\widetilde X$ is a
product of root constructions on each $D_i$ and is a smooth proper
Deligne-Mumford relative surface over $S$.

\begin{defn} The stack of \emph{blt Azumaya algebras\/} is the stack
  $\azumaya_{\widetilde X/S}^\alpha$ whose objects over $T$ are Azumaya
  algebras $\ms A$ on $\widetilde X_T$ such that for every geometric point
  $t\to T$ the fiber $\ms A_t$ is a blt Azumaya algebra with Brauer class
  $\alpha_t$.

\end{defn}

\begin{prop} The stack $\azumaya_{\widetilde X/S}^\alpha$ is an Artin stack
  locally of finite presentation over $S$.  \end{prop} \begin{proof} By the
  main result of \cite{MR2753611}, we know that the stack of all $S$-flat
  coherent algebras on $\widetilde X$ is an Artin stack locally of finite
  presentation on $S$.  The locus of Azumaya algebras is open, as is the locus
  where the type at each $x_i$ is $e_i$.  Finally, the proper and smooth base
  change theorem in \'etale cohomology shows that the locus on which the fibers
  have Brauer class $\alpha$ is clopen.  \end{proof}

\section{Relations among the moduli problems} \label{sec:comparison-result}

\subsection{Pushforwards of Azumaya families are na\"ive families}
\label{sec:pushf-azum-famil}

Let $\ms A$ be a family in $\azumaya_{\widetilde X/S}^\alpha$ over a base $T$.
The pushforward morphism $\pi:\widetilde X\to X$ yields a sheaf of algebras
$A:=\pi_\ast\ms A$.

\begin{prop}\label{P:can push forward} The algebra $A$ described above is a
  family in $\naive_{X/S}^\alpha$.  \end{prop} \begin{proof} First, since
  $\widetilde X$ is tame and $\ms A$ is $T$-flat and coherent, we know that $A$
  is also $T$-flat and coherent, and that the formation of $A$ is compatible
  with base change on $T$.  Thus, to show that $A$ is a family in
  $\naive_{X/S}^\alpha$, it suffices to assume that $T$ is the spectrum of an
  algebraically closed field $K$.  Since $\widetilde X\to X$ is an isomorphism
  over a dense open subset, we know that $A$ is generically Azumaya with Brauer
  class $\alpha$.  By Proposition \ref{P:push forward get terminal} we have
  that $A$ is terminal, and Assumption \ref{A:ass} implies that any terminal
  order is maximal, completing the proof.  \end{proof}

Pushforward along $\pi$ thus defines a $1$-morphism of stacks

$$\Phi:\azumaya_{\widetilde X/S}^\alpha\to\naive_{X/S}^\alpha.$$ This morphism
will be the object of study for the rest of this section.  In particular, we
will show that it is a proper bijection that is not in general surjective on
tangent spaces.  This thus realizes $\azumaya_{\widetilde X/S}^\alpha$ as
something between $\naive_{X/S}^\alpha$ and its normalization.  We are not sure
what normality properties $\azumaya_{\widetilde X/S}^\alpha$ enjoys, but it is
likely that it can be arbitrarily bad (although one might hope for
stabilization as one varies discrete parameters like the second Chern class).

\subsection{Na\"\i ve families over complete dvrs and reflexive blt Azumaya
algebras} \label{sec:local-struct-kolla}

Let $R$ be a complete dvr over $S$ with uniformizer $t$ and algebraically
closed residue field $k$ and let $A\in\naive_{X/S}^\alpha(R)$.  In this section
we will show that locally on $X_R$ the family $A$ comes from a reflexive
Azumaya algebra over a stack with $A_{n-1}$-singularities and coarse moduli
space $X_R$.  We will use this in Section \ref{sec:properness-phi} to show that
$\Phi$ satisfies the valuative criterion of properness.

Write $\widebar X=X[D^{1/n}]$, in the notation of Section 3.B of
\cite{MR2753611}.  This is a stack with coarse moduli space $X$ that may be
locally described as follows: at a crossing section of two components $D_1$ and
$D_2$ of $X$ with local equations $t_1=0$ and $t_2=0$, the stack $\widebar X$
is given by taking the stack-theoretic quotient for the action of $\m_n$ on
$\ms O[w_1,w_2]/(w_1^n-t_1,w_2^n-t_2)$ given by $\zeta\cdot(w_1,w_2)=(\zeta
w_1,\zeta^{-1}w_2)$.  Since $D$ has relative normal crossings, we see that
$\widebar X$ has flat families of $A_{n-1}$-singularities in fibers.

As in Section \ref{sec:globalization}, we have a smooth stack $\widetilde X$
dominating $\widebar X$.

We will prove the following local structure theorem in this section, and then
study reflexive Azumaya algebras on $\widebar X$ in the following section.

\begin{prop}\label{P:local structure} Let $\spec R\to S$ be a dvr over $S$.
  Any algebra in $\naive_{X/S}^\alpha(R)$ is the pushforward from $\widebar X$
  of a unique reflexive blt Azumaya algebra on $\widetilde X_R$ with Brauer
  class $\alpha$.  \end{prop} \begin{proof} Let $A\in\naive_{X/S}^\alpha(R)$.
  By Proposition \ref{P:push forward get terminal}, the generic fiber $A_\eta$
  is the pushforward of an Azumaya algebra $\ms A_\eta$ on $\widetilde X_\eta$.
  Since $\widetilde X\to\widebar X$ is relatively tame, we see that the
  pushforward of $\ms A_\eta$ to $\widebar X$ is a reflexive blt Azumaya
  algebra $\widebar A_\eta$ that pushes forward to $A_\eta$.

  The morphisms $\widetilde X_R\to\widebar X_R\to X_R$ are isomorphisms over
  the generic point of the closed fiber of $X_R$.  Moreover, the order $A$ is
  Azumaya in a neighborhood of that point, and all of the orders and Azumaya
  algebras described so far are contained in the localization $B$ of $A$ at
  this point.  

  \begin{lem}\label{L:local determination} Let $Z$ be an integral $S_2$
    Noetherian Deligne-Mumford stack and $A$ a finite-dimensional
    $\kappa(Z)$-algebra.  Suppose for each codimension $1$ point $z$ there is
    given a maximal order $B_z\subset A$ over the local ring $\ms O_{Z,z}$.
    Then there is at most one maximal order $B$ over $Z$ such that $B\tensor\ms
    O_{Z,z}=B_z\subset A$.  \end{lem} \begin{proof} Given two such maximal
    orders $B$ and $B'$, consider the algebra $B'':=B\cap B'$.  Since $B$ and
    $B'$ are $S_2$, we have that $B''$ is also $S_2$.  Since $B''$ is $S_2$ and
    maximal in codimension $1$ it is maximal. By hypothesis, the inclusions
    $B''\subset B$ and $B''\subset B'$ are isomorphisms are all codimension $1$
    points.  Thus, $B''\to B$ and $B''\to B'$ are isomorphisms, as desired.
  \end{proof}

  Now let $\widebar A$ be any reflexive extension of $\widebar A_\eta$ that
  localizes to $B$.  We see that the pushforward of $\widebar A$ is a maximal
  order agreeing with $A$ in the generic fiber and at the generic point of the
  closed fiber, and thus at all codimension $1$ points.  Applying Lemma
  \ref{L:local determination}, we conclude that $\widebar A$ pushes forward to
  $A$, as desired.  \end{proof}

\subsection{Local structure of reflexive Azumaya algebras on families of
rational double points} \label{sec:refl-azum-algebr}

In this section we will analyze the local structure of reflexive Azumaya
algebras on $\widebar X$.

Let $R$ be a complete dvr with uniformizer $t$ and algebraically closed residue
field $k$ of characteristic $0$.  Let $Z:=\spec B\to\spec R$ be a smooth
relative affine surface and $D_1,D_2\subset Z$ smooth relative curves whose
intersection $S:=D_1\cap D_2$ is isomorphic to the scheme-theoretic image of a
section of $Z/R$.  Replacing $Z$ with an open subscheme containing $S$ if
necessary, we may assume that $D_i$ is the vanishing locus of a global function
$t_i\in\Gamma(Z,\ms O_Z)$, $i=1,2$.  Let $Z'=\spec B[w]/(w^n-t_1t_2)$ be the
cyclic cover branched along $D_1\cup D_2$; there is a section $\sigma:R\simto
S'\subset Z'$ lifting $S$.  There is a stack $\ms Z$ with coarse moduli space
$Z'$ given by taking the quotient of $\spec B[w_1,w_2]/(w_1^n-t_1,w_2^n-t_2)$
by the action of $\m_n$ in which $\zeta\cdot(w_1,w_2)=(\zeta
w_1,\zeta^{-1}w_2)$.  The natural map $\ms Z\to Z'$ is an isomorphism away from
the singular locus $S'$.

Write $z\in Z'$ for the closed point of $S'$, and let $Y'=\spec \ms
O_{Z',z}^{\text{\rm hs}}$ and $\ms Y'=Y\times_{Z'}\ms Z$ be the Henselizations
of $Z'$ and $\ms Z$ at $z$.  Because $R$ is strictly Henselian, there is a
section $T\subset Y\to\spec R$ lying over $S'$.  Finally, let $Y$ be the
Henselization of $Y'$ along $T$ and let $\ms Y=\ms Y'\times_{Y'}Y$, with
$\pi:\ms Y\to Y$ the natural map.  We have that $(\ms Y\times_YT)_{\text{\rm
red}}$ is isomorphic to $\B\m_{n,T}$.  Write $U=Y\setminus T$; this is in fact
the regular locus of $Y$, and it has regular geometric fibers over $R$.  Note
that, as a limit of Henselian local schemes, $Y$ is itself still a Henselian
local scheme.

\begin{lem}\label{L:brtriv} The Brauer group $\Br(U)$ is trivial.  \end{lem}
\begin{proof} By purity, we have that $\Br(U)=\Br(\ms Y)$, so it suffices to
  show that the latter vanishes.  Since $Y$ is Henselian along $T$, we have by
  the usual deformation arguments that $\Br(\ms Y)=\Br(\B\m_{n,T})$, so it
  suffices to show that this last group is trivial.

  Consider the projection $\pi:\B\m_{n,T}\to T$.  The Leray spectral sequence
  yields $$\H^p(T,\R^q\pi_\ast\G_m)\Rightarrow\H^{p+q}(\B\m_{n,T},\G_m).$$  We
  know by \S 4.2 of \cite{paiitbgoaas} that $\R^2\pi_\ast\m_n=0$ and
  $\R^1\pi_\ast\G_m=\Z/n\Z$.  Since $R$ is Henselian with algebraically closed
  residue field we have that $\H^1(T,\Z/n\Z)=0$.  The sequence of low degree
  terms then shows that the pullback map $\H^2(T,\G_m)\to\H^2(\B\m_{n,T},\G_m)$
  is an isomorphism.  But, again because $R$ is Henselian with algebraically
  closed residue field, we know that $\H^2(T,\G_m)=\Br(T)=0$.  \end{proof}

\begin{cor}\label{C:sheafy} A reflexive Azumaya algebra on $Y$ has the form
  $\send(M)$, where $M$ is a reflexive $\ms O_Y$-module.  \end{cor}
\begin{proof} Let $\ms A$ be a reflexive Azumaya algebra.  By Lemma
  \ref{L:brtriv} we know that $\ms A|_U\cong\send(V)$ with $V$ a locally free
  coherent sheaf on $U$.  If $M$ is the unique reflexive coherent extension of
  $V$ then $\send(M)$ is reflexive and isomorphic to $\ms A$ in codimension
  $1$, whence $\ms A\cong\send(M)$.  \end{proof}

\begin{prop} Suppose $\ms A$ is a reflexive Azumaya algebra of degree $r$ on
$Y$ such that the restriction $\ms A\tensor k$ is a reflexive Azumaya algebra
on $Y\tensor k$.  Then \begin{enumerate} \item $\ms A\cong\send(M)$ with $M$ a
      direct sum of indecomposable reflexive $\ms O_Y$-modules of rank $1$;
    \item there is a blt Azumaya algebra $\ms B$ on $\ms Y$ such that $\ms
      A=\pi_\ast\ms B$.  \end{enumerate} \end{prop} \begin{proof} By assumption
  we have that $\ms A\tensor k\cong\send(V)$ with $V$ a reflexive $\ms
  O_{Y\tensor k}$-module.  But $Y\tensor k$ is the Henselization of an
  $A_{n-1}$-singularity, so we know that $V$ decomposes as a direct sum of
  reflexive modules of rank $1$ by the McKay correspondence \cite{MR769609}.
  This gives rise to a full set of idempotents $e_j\in\ms A(Y\tensor k)$,
  $j=1,\ldots,r$.  Since $Y$ is Henselian, these idempotents lift to global
  sections $\widetilde e_j$ of $\ms A$.  By Corollary \ref{C:sheafy} we have
  that $\ms A\cong\send(M)$.  The idempotents $\widetilde e_j$ decompose $M$ as
  a direct sum of submodules of rank $1$.  Since $M$ is reflexive, each of
  these summands is reflexive, proving the first statement.

  To prove the second statement, note that a reflexive sheaf of rank $1$ on $Y$
  is the pushforward along $\pi$ of a unique invertible sheaf on $\ms Y$.
  Thus, $M$ is isomorphic to $\pi_\ast N$ for some locally free sheaf $\ms V$
  on $\ms Y$.  The Azumaya algebra $\ms B=\send(N)$ has reflexive pushforward
  that is canonically isomorphic to $\ms A$ over $U$, whence $\ms
  A\cong\pi_\ast\ms B$, as desired.  \end{proof}

\subsection{Proof that $\Phi$ is a proper bijection} \label{sec:properness-phi}

In this section we show that $\Phi:\azumaya_{\widetilde
X/S}^\alpha\to\naive_{X/S}^\alpha$ is a proper morphism.  Since it is already
locally of finite presentation and bijective, it suffices to show the following
valuative criterion.

\begin{prop} If $R$ is a complete dvr over $S$ then any na\"\i ve family $A$ on
  $X_R$ has the form $\pi_\ast\ms A$, where $\ms A$ is an Azumaya family on
  $\widetilde X_R$.  \end{prop} \begin{proof} By Propostion \ref{P:local
  structure}, we know that $A$ is the pushforward of a family $\ms B$ of
  reflexive Azumaya algebras on $\widebar X$.  It thus suffices to show that
  $\ms B\cong p_\ast\ms A$, where $p:\widetilde X\to\widebar X$ is the natural
  morphism.

  Let $V\subset\widebar X$ be the smooth locus of $\widebar X/S$.  By
  construction there is a natural diagram $$\xymatrix{ & \widetilde X\ar[dd] \\
  V\ar[ur]^i\ar[dr]_j & \\ & \widebar X}$$ in which the diagonal arrows are
  fiberwise dense open immersions whose complements have codimension two in
  each geometric fiber.  By the theory of hulls \cite{kollar}, we have that the
  adjunction map $\ms B\to j_\ast\ms B_V$ is an isomorphism.  Since $pi=j$, to
  prove the result it suffices to prove that $i_\ast\ms B_V$ is an Azumaya
  algebra on $\widetilde X$.

  This latter statement is \'etale local on $X$, so we may replace $X$ by the
  local Henselian scheme $Y$ of section \ref{sec:refl-azum-algebr}.  In this
  case we have that $\ms B$ is isomorphic to $p_\ast\ms A$ for some Azumaya
  algebra $\ms A$. The algebra $i_\ast\ms B_V$ is thus isomorphic to $i_\ast\ms
  A_V$, and so it suffices to show that the adjunction map $a:\ms A\to
  i_\ast\ms A_V$ is an isomorphism.  But the stack $\widetilde X$ is regular
  and $\ms A$ is locally free, so $a$ is an isomorphism if and only if it is an
  isomorphism in codimension $1$.  Since $V$ has codimension $2$, we know that
  $a$ is an isomorphism at every codimension $1$ point, and the result follows.
\end{proof}

\subsection{Proof that $\Phi$ need not be an isomorphism}
\label{sec:moduli-problems-are}

In this section we prove that the map $\azumaya_{\widetilde
X/S}^\alpha\to\naive_{X/S}^\alpha$ need not be an isomorphism by exhibiting
examples for which the map on tangent spaces is not surjective.

Let $D\subset X$ be a smooth divisor in a projective surface such that
\begin{enumerate} \item there is an infinitesimal deformation $\mc D\subset
      X\tensor k[\eps]$ for which $\ms O_{X\tensor k[\eps]}(\mc D-D\tensor
      k[\eps])$ is non-torsion in $\Pic(X_{k[\eps]})$; \item there is a blt
        Azumaya algebra $\ms A$ on $\widetilde X$ for which the pushforward
        $\pi_\ast\ms A$ is a maximal order on $X$ of period $n$ and
        $\H^2(\widetilde X,\ms A/\ms O_{\widetilde X})=0$.  \end{enumerate}

Write $\widetilde X'=X_{k[\eps]}\{\mc D^{1/n}\}$ and (by abuse of notation)
$\pi:\widetilde X'\to X_{k[\eps]}$ for the projection to the coarse moduli
space.  By deforming $\ms A$ to $\widetilde X'$ we will make a tangent vector
to $\naive_{X/S}$ that does not lie in the image of the tangent map to
$\azumaya_{\widetilde X/S}$.

\begin{prop}\label{not-the-same} There is a deformation $\ms A'$ of $\ms A$ to
  a blt Azumaya algebra on $\widetilde X'$ such that the resulting object
  $A'=\pi_\ast\ms A'$ of $\naive_{X/S}(k[\eps])$ is not in the image of
  $\azumaya_{\widetilde X/S}$.  \end{prop} \begin{proof} The key to the proof
  is to relate the dualizing bimodule of $A'$ to the divisor class $\mc D$.  

  \begin{lem}\label{L:dualizing-computation} Given $\widetilde X'$, $\ms A'$
    and $A'$ as above, there is a natural isomorphism $$ \omega_{A'}^{\otimes
    n} \cong A' \otimes_{{\ms O}_{X_{k[\eps]}}} \omega^n_{X_{k[\eps]}/k[\eps]}
    ((n - 1)\mc D).  $$ \end{lem} Let us briefly accept Lemma
  \ref{L:dualizing-computation} and see how to complete the proof of
  Proposition \ref{not-the-same}.  

  \begin{lem}\label{L:det-norm} The pullback map $\Pic(X_{k[\eps]})\to\Pic(A')$
    is injective modulo torsion.  \end{lem} \begin{proof} The reduced norm
    defines a sequence $\ms O^\times\to(\ms A')^\times\to\ms O^\times$ such
    that the composition is raising to the $n$th power (and thus surjective in
    the \'etale topology).  Applying the \'etale $\H^1$ functor we see that the
    map $\H^1(X,\ms O_X^\times)\to\H^1(X,(A')^\times)$ is injective modulo
    $n$-torsion.  Finally, we note that invertible $A'$-bimodules are
    classified by the latter cohomology group.  \end{proof}

  If $A'$ is in the image of $\azumaya_{\widetilde X/S}(k[\eps])$, the
  analogous computation with $D_{k[\eps]}$ in place of $\mc D$ would yield an
  isomorphism between the bimodules $\omega_{A'}^{\tensor n}$ and
  $A'\tensor_{\ms
    O_{X_{k[\eps]}}}\omega^n_{X_{k[\eps]}/k[\eps]}((n-1)D_{k[\eps]})$.
    Applying Lemma \ref{L:det-norm}, we conclude that $\ms O(\mc
    D-D_{k[\eps]})$ is torsion in $\Pic(X_{k[\eps]})$, contrary to our original
    hypothesis.  \end{proof}

It remains to prove Lemma \ref{L:dualizing-computation}.

\begin{proof}[Proof of Lemma \ref{L:dualizing-computation}] To simplify
  notation, write $X'=X_{k[\eps]}$ and write $\omega_{X'}$ for the relative
  dualizing sheaf over $k[\eps]$.  Recall that the dualizing bimodule is given
  by the sheaf $\shom_{\ms O_{X'}}(A',\omega_{X'})$.  Writing $A'=\pi_\ast\ms
  A'$ and using duality, we have isomorphisms of bimodules

  \begin{align*} \omega_{A'}=\shom_{\ms O_{X'}}(\pi_\ast\ms
    A',\omega_{X'})&=\pi_\ast\shom(\ms A',\pi^!\omega_{X'})\\
    &=\pi_\ast\shom(\ms A',\omega_{\widetilde X'})=\pi_\ast(\ms
    A'\tensor\omega_{\widetilde X'}).  \end{align*}

  As an $\ms A'$-bimodule, we have that $(\ms A'\tensor\omega_{\widetilde
  X'})^{\tensor n}\cong\ms A'\tensor\omega_{X'}^{\tensor n}((n-1)\mc D)$.  To
  see this, note that we can locally write $\widetilde X'=[\spec\ms
    O_{X'}[z]/(z^n-t)/\m_n]$, where $t$ is a local equation for $\mc D$;
    computing the relative differentials immediately yields the result.

  There is a natural map $$(\pi_\ast(\ms A'\tensor\omega_{\widetilde
  X'}))^{\tensor n}\to\pi_\ast((\ms A'\tensor\omega_{\widetilde X'})^{\tensor
  n})$$ giving rise (using the computation of the preceding paragraph) to a map
  $$\phi:\omega_{A'}\to A'\tensor\omega_{X'}^{\tensor n}((n-1)D)$$ that we wish
  to show is an isomorphism.  

  Note that an \'etale-local model for $X',A',\ms A'$ around a closed point of
  $X'$ is given by the trivial family whose fiber is the standard cyclic
  algebra.  Thus, to prove that $\phi$ is an isomorphism it suffices to prove
  it for the local constant family, and thus (by compatibility with pullback)
  for the local family over a smooth surface over $k$.  But this is Proposition
  5 of \cite{thong}.  \end{proof}

To give a concrete example, let $X=E\times E$ for a smooth projective curve of
genus $1$ over an algebraically closed field of characteristic $0$ and let
$D=D_1+D_2$ be the sum of two disjoint closed fibers of the second projection.
Let $E'\subset E$ be the complement of the image of $D$ under $\pr_2$.  There
is a finite covering $C\to E'$ of degree $2$ that is totally ramified at both
points of $E\setminus E'$, giving a class $\alpha\in\H^1(E',\Z/2\Z)$.  Choosing
any $\beta\in\H^1(E,\mu_2)$ we can form the class
$\pr_2^\ast\alpha\cup\pr_1^\ast\beta\in\H^2(E\times E',\m_2)$, giving a Brauer
class $\gamma\in\Br(k(E\times E))$.  Elementary computations show that the
ramification extension of this Brauer class on each component of $D$ is given
by the class of $\beta$, so that maximal orders must be hereditary along $D$.

Any non-constant infinitesimal deformation of $D$ (e.g., that induced by moving
along $E$) will give a $\mc D$ as in the statement of Proposition
\ref{not-the-same}.  It remains to show that there is an unobstructed Azumaya
algebra on the stack $\widetilde X\to E\times E$ branched over $D$.  Since
$\Br(\widetilde X)=\Br'(\widetilde X)$, there is certainly some Azumaya algebra
in that class.  Producing one that is unobstructed is a standard argument that
can be found written out for projective surfaces in Proposition 3.2 of
\cite{MR2060023}.  We omit the details.

\begin{remark} The construction given here also shows that fixing the Brauer
  class to be $\alpha$ \'etale-locally on the base of families in Definition
  \ref{D:naive} does not ameliorate the situation, as any infinitesimal
  deformation of the class of $\alpha$ on $E\times E'$ is constant.
\end{remark}

\section{Generalized Azumaya algebras on $\widetilde X$}
\label{sec:gener-azum-algebr} \setcounter{lem}{0} \numberwithin{lem}{section}

In this section we will suppose that $S=\spec k$ is the spectrum of an
algebraically closed field.  We will compactify the stack $\azumaya_{\widetilde
X/S}^\alpha$ and show that this compactification has a relative virtual
fundamental class when $\alpha$ has order $n$ on each geometric fiber of $X/S$.
The constructions described here are almost identical to those in
\cite{MR2579390}.  By the proper and smooth base change theorems in \'etale
cohomology, any family in $\azumaya_{\widetilde X/S}^\alpha$ defines a section
of the finite constant sheaf (scheme!) $\R^2f_\ast\m_n$, where $f:\widetilde
X\to S$ is the structural morphism, giving a morphism of stacks
$$c:\azumaya_{\widetilde X/S}^\alpha\to\R^2f_\ast\m_n.$$  Thus, to compactify
$\azumaya_{\widetilde X/S}^\alpha$ we will compactify each fiber.

Write $\widebar\alpha\in\H^2(\widetilde X,\m_n)$ for a lift of $\alpha$ via the
Kummer sequence.  The fiber of $c$ over $\widebar\alpha$ will be denoted
$\azumaya_{\widetilde X/S}^{\widebar\alpha}$.  Let $p:\ms X\to\widetilde X$ be
a $\m_n$-gerbe representing the class $\widebar\alpha$.  That there is such an
Artin stack is discussed in Section 2.4 of \cite{paiitbgoaas}.  In Sections 2.2
and 2.3 of \cite{paiitbgoaas} or in \cite{MR2388554} the reader will also find
a discussion of the theory of $\ms X$-twisted sheaves in connection with the
Brauer group.

\begin{defn}\label{regular twisted sheaf} A torsion free $\ms X$-twisted sheaf
  $\ms F$ is \emph{blt\/} if the $\ms O_{\widetilde X}$-algebra
  $p_\ast\send(\ms F)$ is blt on the Azumaya locus.  \end{defn}

Let $\Sh_{\ms X}$ denote the stack of torsion free blt $\ms X$-twisted sheaves
of rank $n$ with trivial determinant.  The basic result on the stack $\Sh_{\ms
X}$ is the following.

\begin{prop}\label{stack of twisted sheaves is artin} The stack $\Sh_{\ms X}$
  is an Artin stack locally of finite presentation over $S$.  Moreover,
  $\Sh_{\ms X}$ is a $\G_m$-gerbe over an algebraic space $\mSh_{\ms X}$ with
  proper connected components.  \end{prop} \begin{proof} This proven in
  Sections 3 and 4 of \cite{Lieblich20114145}, once we note that any torsion
  free $\ms X$-twisted sheaf of rank $n$ is automatically stable when the
  Brauer class has period $n$.  \end{proof}

Let $\Sh^f_{\ms X}$ denote the locus of locally free $\ms X$-twisted sheaves.
The morphism $\ms V\mapsto p_\ast\send(\ms V)$ defines a morphism of stacks
$e:\Sh^f_{\ms X}\to\azumaya_{\widetilde X/S}^{\widebar\alpha}$.  

\begin{lem}\label{end epi} The morphism $e$ is an epimorphism of stacks.
\end{lem} \begin{proof} Since both stacks are locally of finite presentation,
  it suffices to prove that if $S$ is strictly Henselian and $A$ is an Azumaya
  on $\widetilde X_S$ with Brauer class $\alpha$, then $A$ is of the form
  $p_\ast\send(\ms V)$ for $\ms V$ a blt locally free $\ms X_S$-twisted sheaf.

  This follows immediately from Giraud's description of the cohomology class in
  $\H^2(\widetilde X,\m_n)$ associated to $A$: one takes the stack of
  isomorphisms $\send(V)\simto A$ with $V$ locally free with trivialized
  determinant $\det V\simto\ms O$.  This is a $\m_n$-gerbe $\ms X$, and the
  sheaves $V$ glue to give an $\ms X$-twisted sheaf of rank $n$ with trivial
  determinant.  \end{proof}

Let $G=\Pic_{\widetilde X/S}[n]$ be the (finite) $n$-torsion subgroupscheme of
the relative Picard scheme.  Given an invertible sheaf $\ms L$ with a
trivialization $\ms L^{\tensor n}\simto\ms O_{\widetilde X}$ there is an
induced $1$-morphism $\tensor\ms L:\Sh_{\ms X}\to\Sh_{\ms X}$.

\begin{lem}\label{action of torsion} The morphisms $\tensor\ms L$ defined above
  as $\ms L$ ranges over a set of representatives for $G$ define an action
  $G\times\mSh_{\ms X}\to\mSh_{\ms X}$.  \end{lem} \begin{proof} Given an
  invertible sheaf $\ms L$ with a trivialization $\ms L^{\tensor n}\simto\ms O$
  and a torsion free sheaf $\ms F$ of rank $n$ with a trivialization $\det\ms
  F\simto\ms O$, we get a trivialization $$\det(\ms F\tensor\ms
  L)\simto\det(\ms F)\tensor\ms L^{\tensor n}\simto\ms O\tensor\ms O\simto\ms
  O.$$ This map induces the action.  \end{proof}

\begin{prop}\label{azumaya open part} The morphism $e:\ms V\mapsto
  p_\ast\send(\ms V)$ induces an isomorphism of stacks $$[\mSh^f_{\ms
  X}/G]\simto\azumaya_{\widetilde X/S}^{\widebar\alpha}.$$ \end{prop}
\begin{proof} Via the morphism $e$ the scalar multiplication action on $\ms V$
  is sent to the trivial action on $p_\ast\send(\ms V)$ so that $e$ factors
  through an epimorphism of stacks $\eps:\mSh_{\ms X}^f\to\azumaya_{\widetilde
  X/S}^\alpha$.  It follows from the Skolem-Noether theorem that any
  isomorphism $p_\ast\send(\ms V)\simto p_\ast\send(\ms V')$ comes from an
  isomorphism $\ms V\simto\ms V'\tensor L$ for some invertible sheaf $L$, and
  that any invertible sheaf $L$ induces a canonical isomorphism
  $p_\ast\send(\ms V)\simto p_\ast\send(\ms V\tensor L)$.

  Since $G$ acts by twisting by invertible sheaves, the morphism $\eps$ factors
  through the quotient as $\widebar\eps:[\mSh^f_{\ms
  X}/G]\to\azumaya_{\widetilde X/S}^{\widebar\alpha}$.  On the other hand,
  suppose given an isomorphism $p_\ast\send(\ms V)\simto p_\ast\send(\ms W)$.
  By the above remark, we have that there is an invertble sheaf $M$ and an
  isomorphism $\ms V\simto\ms W\tensor M$.  Taking determinants gives an
  isomorphism $\det V\simto\det W\tensor M^{\tensor n}$.  Via the isomorphisms
  $\det V\simto\ms O$ and $\det W\simto\ms O$ we get a canonical isomorphism
  $M^{\tensor n}\simto O$, displaying $\ms W$ as the image of $\ms V$ under
  $\tensor M$.  This shows that $\widebar\eps$ is a monomorphism, showing that
  it is an isomorphism.  \end{proof}

We are now ready to compactify $\azumaya_{\widetilde X/S}^{\widebar\alpha}$ so
that there is a virtual fundamental class. We only sketch the idea here,
deferring a fuller treatment to a subsequent paper.

\begin{prop}\label{P:virt class} The stack $[\mSh_{\ms X}/G]$ carries a virtual
  fundamental class and compactifies $\azumaya_{\widetilde
  X/S}^{\widetilde\alpha}$.  
\end{prop} 
\begin{proof}[Sketch of proof.] Exactly as in Proposition
  6.5.1.1 of \cite{MR2579390}, there is a natural virtual fundamental class on
  $\mSh_{\ms X}$ with perfect obstruction theory given by the complex of
  traceless homomorphisms $\R p_\ast\R\send(\ms V,\omega_{\widetilde
  X/S}\ltensor\ms V)_0$.  Taking the trace of this obstruction theory as in
  Section 6.5.2 of \cite{MR2579390} yields a perfect obstruction theory on the
  quotient $[\mSh_{\ms X}/G]$, as desired.  
\end{proof}

\bibliography{kulbib}{}
\bibliographystyle{plain}

\end{document}